\documentclass[12pt,twoside]{article}
\usepackage{amssymb}
\usepackage{amsmath}
\usepackage{amsthm}
\usepackage{amscd}
\usepackage{amstext}
\usepackage{vmargin}
\setpapersize{A4}
\setmarginsrb{3.0cm}{1.5cm}{3.0cm}{3cm}{1cm}{.5cm}{2cm}{2cm}
\numberwithin{equation}{section}
\usepackage{amssymb}
\usepackage{amsmath}
\usepackage{amsfonts}
\usepackage{amsthm}
\usepackage{natbib}  
\usepackage[latin1]{inputenc}
\usepackage[T1]{fontenc}
\usepackage{enumerate}
\usepackage[american]{babel}
\usepackage{graphicx}
\usepackage{bbm}
\usepackage[mathscr]{euscript}
\usepackage{stmaryrd}
\usepackage{soul}
\usepackage{hyperref}
\usepackage{xspace}
\usepackage{mathtools}
\mathtoolsset{showonlyrefs}
\usepackage{url}
\usepackage[ps,all,dvips]{xy}
\SelectTips{cm}{12}
\CompileMatrices

\usepackage[usenames]{color}

  \makeatletter
 \useshorthands{"}
 \defineshorthand{"-}{\nobreak-\bbl@allowhyphens}
 \makeatother

\newcommand{\mathset}[1]{\mathbbm{#1}}

\newcommand{\F}{\mathcal{F}}

\newcommand{\N}{\mathset{N}}

\newcommand{\levy}{L\'evy}

\newcommand{\R}{\mathset{R}}
\newcommand{\Q}{\mathset{Q}}

\newcommand{\pd}{\overline{\mathcal{P}}^d_\h}

\newcommand{\wpd}{weak\text{-}\overline{\mathcal{P}}^d_\h}

\newcommand{\h}{\mathcal{H}}
\newcommand{\xT}{(X_t)_{t\in T}}

\newcommand{\pid}{\overline \Pi_\G^d}
\newcommand{\s}{\mathcal{S}}

\newcommand{\B}{\mathcal{B}}

\newcommand{\G}{\mathcal{G}}

\newcommand{\dist}{\stackrel{\mathcal{D}}{=}}
\newcommand{\morf}[4][\to]{ #2 \colon #3 #1 #4}
\newcommand{\inv}{^{-1}}

\newcommand{\fre}{Fr\'echet}

\newcommand{\cdist}{\stackrel{\mathcal{D}}{\rightarrow}}
\newcommand{\ph}{\mathcal{P}}

\makeatletter
\newcommand{\fancybreak}{\@ifstar{\@sfbreak}{\@fbreak}}
\newcommand{\@fbreak}[1]{\par
   \penalty -100
   \noindent\parbox{\linewidth}{\centering #1}\null
   \penalty -20
   \@afterindentfalse
   \@afterheading}
\newcommand{\@sfbreak}[1]{\par
   \penalty -100
   \noindent\parbox{\linewidth}{\centering #1}\null
   \penalty -20
   \@afterindenttrue
   \@afterheading}
\makeatother

\newcommand{\abs}[2][]{#1\lvert #2#1\rvert}
\newcommand{\norm}[2][]{#1\lVert #2#1\rVert}

\renewcommand{\phi}{\varphi}
\renewcommand{\epsilon}{\varepsilon}

\theoremstyle{plain}
\newtheorem{lemma}{Lemma}[section]
\newtheorem{proposition}[lemma]{Proposition}
\newtheorem{theorem}[lemma]{Theorem}
\newtheorem{corollary}[lemma]{Corollary}
\theoremstyle{definition}

\theoremstyle{definition}
\newtheorem{definition}[lemma]{Definition}
\theoremstyle{remark}
\newtheorem{remark}[lemma]{Remark}

\newcommand{\devnull}[1]{}

\usepackage{fancyhdr}
\title{Integrability of Seminorms}
  \author{Andreas Basse\thanks{Department of Mathematical Sciences,
 University of Aarhus,\newline  Ny Munkegade, DK-8000 \AA rhus
   C, Denmark. E-mail: basse@imf.au.dk} }
\begin{document}\maketitle

\begin{abstract}
  We study integrability and equivalence of $L^p$-norms of polynomial
  chaos  elements. Relying on known results for Banach space valued
  polynomials, a simple technique is presented to obtain
  integrability results for random
  elements that are not necessarily limits of Banach space valued
  polynomials. This enables us to prove integrability results for a
  large class of seminorms of stochastic processes and  
 to answer,  partially, a question raised by C.\ Borell (1979,
 Séminaire de Probabilités, XIII, 1--3). 
\end{abstract}

\section{Introduction}
Let $T$ denote a countable set, $X=\xT$  a stochastic process and  
$N$ a seminorm on $\R^T$. This paper focuses on 
  integrability and equivalence of $L^p$-norms of $N(X)$ in the case
  where $X$ is a weak chaos process; see Definition~\ref{def_svag}. Of
  particular interest is the supremum and the $p$-variation norm given
  by \begin{align}
  \label{eq:57}
   N(f)=\sup_{t\in T}\abs{f(t)}\quad \text{and}\quad
 N(f)=\sup_{n\geq 1}   \Big(\sum_{i=1}^{k_n}
  \abs{f(t^{n}_i)-f(t^{n}_{i-1})}^p\Big)^{1/p},\ p\geq 1, \quad 
\end{align} for $f\in \R^T$.  In the $p$-th variation case we assume moreover
$T=[0,1]\cap\Q$ and  $\pi_n=\{0=t^{n}_0<\dots<t^{n}_{k_n}=1\}$ are
nested subdivisions of $T$ satisfying $\cup_{n=1}^\infty\pi_n=T$. 
Note that if $N$ is given by \eqref{eq:57}, $B=\{x\in \R^T:N(x)<\infty\}$ and 
$\norm{x}=N(x)$ for $x\in B$, then  
$(B,\norm{\,\cdot\,})$ is a non-separable Banach space when
$T$ is infinite.

Our results partly unify and partly extend known results in this
area. For relations to the literature see Subsection~\ref{example}. 
We note, however,  that in the setting of the present paper we are
able to treat Rademacher chaos processes of arbitrary order as well as
infinitely divisible integral processes as in \eqref{eq:9} below.

\subsection{Chaos Processes and Condition $C_q$}
\label{setup}
Let $(\Omega,\F,P)$ denote a probability space.  When $F$ is a
topological space, a Borel measurable mapping 
$\morf{X}{\Omega}{F}$ is called an $F$-valued random element, however
when $F=\R$, $X$ is, as usual, called a random variable.  For each  $p>0$ and
random  variable $X$  we let
$\norm{X}_p:=E[\abs{X}^p]^{1/p}$, which defines a norm when $p\geq
1$; moreover,  let $\norm{X}_\infty:=\inf\{t\geq 0:P(\abs{X}\leq
t)=1\}$. When $F$ is a Banach
space, $L^p(P;F)$ denotes the space of 
all $F$-valued random elements, $X$, satisfying 
$\norm{X}_{L^p(P;F)}=E[\norm{X}^p]^{1/p}<\infty$. Throughout the
paper $I$ denotes a set and  for all
$\xi\in I$, $\h_\xi$ is a family of 
independent random variables. Set $\h=\{\h_\xi:\xi\in I\}$. 
Furthermore, $d\geq 1$ is a natural number 
and $F$ is a locally convex Hausdorff  
topological vector space
(l.c.TVS) with dual space $F^*$. Following
\citet{Fernique_book}, a  map $N$ from  $F$ into $[0,\infty]$  
is called a  pseudo-seminorm  if for all $x,y\in F$ and $\lambda\in
\R$, we have 
\begin{align}
  \label{eq:48}
  N(\lambda x)=\abs{\lambda} N(x)  \qquad \text{and} \qquad N(x+y)\leq
  N(x)+N(y). 
\end{align}
 For $\xi\in I$ let  $\ph_{\h_\xi}^d(F)$ denote the set of 
 $p(Z_1,\dots,Z_n)$ where $n\geq 1,\ Z_1,\dots,Z_n$ are different
 elements in $\h_\xi$ and $p$ is an $F$-valued tetrahedral polynomial of
       order $d$. Recall that $\morf{p}{\R^n}{F}$  is called an
       $F$-valued tetrahedral polynomial of order $d$ if there exist $
x_0,x_{i_1,\dots,i_k}\in F$ and $l\geq 1$ such that 
\begin{align}
  \label{eq:174}
  p(z_1,\dots,z_n)=x_0+\sum_{k=1}^d\sum_{1\leq i_1<\dots<i_k\leq l} x_{i_1,\dots,i_k}
  \prod_{j=1}^k z_{{i_j}}.
\end{align} 
For each  $F$-valued  random element $X$
we will write 
$X\in \pd(F)$  if  there exist $(\xi_k)_{k\geq
   1}\subseteq I$ and  $X_k\in \ph_{\h_{\xi_k}}^d(F)$ for $k\geq 1$
 such that $X_k\cdist X$. 
    Inspired by \citet{Talagrand}
we introduce the following definition: 
\begin{definition}\label{def_svag}\
  An $F$-valued random element $X$ is said to be a  weak 
chaos element of order $d$ associated with $\h$ if for all $n\geq 1$ and all 
$(x_i^*)_{i=1}^n\subseteq F^*$  we have
$(x_1^*(X),\dots,x_n^*(X))\in\pd(\R^n)$, and  in this case we write  
$X\in weak\text{-}\pd(F)$. Similarly, a real-valued stochastic  
process $\xT$ is said to be a weak chaos
process of  order $d$ associated with $\h$ if for all $n\geq 1$ and
$(t_i)_{i=1}^n\subseteq T$ we have
$(X_{t_1},\dots,X_{t_n})\in\pd(\R^n)$. 
  \end{definition}
\noindent An important example of a weak chaos process of order one
is $(X_t)_{t\in T}$ of the form  
\begin{align}
  \label{eq:9}
  X_t=\int_S f(t,s)\,\Lambda(ds),\qquad t\in T,
\end{align} where $\Lambda$ is an independently
scattered infinitely divisible random measure (or random measure for
short) on some non-empty  
space $S$ equipped with a $\delta$-ring $\s$, and $s\mapsto f(t,s)$ are 
$\Lambda$-integrable deterministic
functions in the sense of  \cite{Rosinski_spec}. To obtain the associated
$\h$ let $I$  be the set of all 
$\xi$ given by $\xi=\{A_1,\dots,A_n\}$ for some $n\geq 1$ and
disjoint sets $A_1,\dots,A_n$ in $\s$,  and let 
\begin{equation}
  \label{eq:46}
  \h_\xi=\{\Lambda(A_1),\dots,\Lambda(A_n)\}\qquad \text{and}\qquad
  \h=\{\h_\xi\}_{\xi\in I}.
\end{equation} Then, by definition of the stochastic integral
\eqref{eq:9} as the limit of integrals of simple functions,
$(X_t)_{t\in T}$ is a weak chaos process of order one 
associated with   $\h$. 

Another example is where $(Z_n)_{n\geq 1}$ is sequence of
independent random variables and
$x(t),x_{i_1,\dots,i_k}(t)\in\R$ are real numbers for which 
\begin{align}
  \label{eq:149}
  X_t=x(t)+\sum_{k=1}^d \sum_{1\leq i_1<\dots <i_k<\infty} x_{i_1,\dots,i_k}(t)
  \prod_{j=1}^k Z_{i_j},
\end{align} exists  in
probability for all $t\in T$; then  $X=(X_t)_{t\in T}$ is a weak chaos process
of order $d$ associated with $I=\{0\}$,  $\h_0=\{Z_n:n\geq
1\}$ and $\h=\{\h_0\}$.

In what follows we shall need  the next  conditions:
\\  \noindent\textbf{Notation, Condition~$C_q$}  
  \begin{itemize}\label{con_q}
  \item For $q\in (0,\infty)$, $\h$ is said to satisfy  $C_q$ if there
    exists    $\beta_1,\beta_2>0$ such that for
    all  $Z\in \cup_{\xi\in I}\h_\xi$ there exists
$c_Z>0$ with $P(\abs{Z}\geq c_Z)\geq \beta_1$ and 
    \begin{align}
      \label{eq:139}
 E[\abs{  Z}^q,\abs{Z}>s]\leq {}&\beta_2 s^q P(\abs{Z}>s),\qquad s\geq c_Z.
\end{align} 
\item \label{con_inf} $\h$ is said to satisfy  $C_\infty$ if 
  $\cup_{\xi\in I}\h_\xi\subseteq  L^1$ and 
  \begin{align}
    \label{eq:118}
  \sup_{\xi\in I}\sup_{Z\in
     \h_\xi}\left(\frac{\norm{Z-E[Z]}_\infty}{\norm{Z-E[Z]}_2}\right)=\beta_3<\infty. 
  \end{align}
  \end{itemize}
  \begin{remark}\label{eqi_norms}
    If $\h$ satisfies $C_q$ for some $q<\infty$ then for all $p\in (0,q)$ we have 
    \begin{equation}
      \label{eq:134}
      \sup_{\xi\in I}\sup_{Z\in
     \h_\xi}\frac{\norm{Z}_q}{\norm{Z}_p}\leq (\beta_2\vee 1)^{1/q}\beta_1^{-1/p}<\infty. 
    \end{equation} This follows by the next two estimates: 
   \begin{align}
      \label{eq:136}
      E[\abs{Z}^q]={}& E[\abs{Z}^q,\abs{Z}>c_Z]+E[\abs{Z}^q,\abs{Z}\leq
      c_Z]\\ \leq {}& \beta_2 c_Z^q  P(\abs{Z}>c_Z)+c_Z^q P(\abs{Z}\leq
      c_Z) \leq (\beta_2\vee 1) c_Z^q
    \shortintertext{and} 
   c_Z^p \beta_1\leq {}& c_Z^p P(\abs{Z}\geq c_Z)\leq E[\abs{Z}^p].
\end{align}
  \end{remark}
For example, when all $Z\in \cup_{\xi\in I} \h_\xi$ have the same
distribution, $\h$ satisfies
 $C_q$ for all $q\in (0,\alpha)$ for $\alpha>0$ if   $x\mapsto P(\abs{Z}>x)$ is
regulary varying with index $-\alpha$,
 by  Karamata's Theorem; see \citet[Theorem~1.5.11]{Bingham}. 
In particular,  if the common distribution is symmetric $\alpha$-stable
for some $\alpha\in (0,2)$ then  $\h$ satisfies
$C_q$ for all $q\in (0,\alpha)$. If the common distribution is
Poisson, exponential, Gamma or Gaussian 
 then   $C_q$ is satisfied for all $q>0$. 
Finally $\h$ satisfies $C_\infty$ if and only if the common
distribution has compact support.

As we shall see in Section~\ref{sec_main}, $C_q$ is crucial in order
to obtain integrability results and equivalence of $L^p$-norms, so  let us
consider some cases where the important example \eqref{eq:9} does or
does not satisfy $C_q$. For this purpose let us introduce the following
distributions: The inverse Gaussian distribution IG$(\mu,\lambda)$ with
 $\mu,\lambda>0$ is the distribution on $\R_+$ with density 
\begin{equation}
  \label{eq:146}
  f(x;\mu,\lambda)=\left[\frac{\lambda}{2\pi
  x^3}\right]^{1/2}e^{-\lambda(x-\mu)^2/(2\mu^2x)},\qquad x>0.
\end{equation} Moreover, the normal inverse Gaussian distribution
NIG$(\alpha,\beta,\mu,\delta)$ with $\mu\in\R,\ \delta\geq 0$, and
$0\leq \beta\leq \alpha$, is symmetric if and only if
$\beta=\mu=0$, and in this case  it has the following density 
\begin{equation}
  \label{eq:147}
  f(x;\alpha,\delta)=\frac{\alpha e^{\delta \alpha} }{\pi \sqrt{1+x^2\delta^{-2}} }
  K_1\left(\delta \alpha (1+x^2\delta^{-2})^{1/2}\right),\qquad x\in\R,
\end{equation} where $K_1$ is the modified Bessel function of the
third kind and index 1 given by $K_1(z)=\frac{1}{2}
\int_0^\infty e^{-z(y+y\inv)/2}\,dy$ for $z>0$. 

For each finite number $t_0>0$, a random measure  $\Lambda$ is said to
be induced by a \levy\ process $Y=(Y_t)_{t\in [0,t_0]}$ 
if $S=[0,t_0]$, $\s=\B([0,t_0])$ and 
$\Lambda(A)=\int_A \,dY_s$ for all $A\in \s$. 
\begin{proposition}\label{pro_levy} Let  $t_0\geq 1$ be a finite number,
  $\Lambda$ a random measure induced by a L\'evy process  
  $Y=(Y_t)_{t\in [0,t_0]}$ and $\h$ be
  given by \eqref{eq:46}. 
  \begin{enumerate}[(i)]
   \item \label{IG} If  $Y_1$ has an IG-distribution, then $\h$
  satisfies $C_q$ if and only if   $q\in (0,\frac{1}{2})$. 
\item \label{NIG} If $Y_1$ has a symmetric NIG-distribution, 
  then $\h$  satisfies $C_q$ if and only if   $q\in (0,1)$. 
\item \label{not_4} If  $Y$ is non-deterministic and has no Gaussian component, then $\h$ 
  does not  satisfy $C_q$ for any $q\geq 2$. In fact, for all
  square-integrable non-deterministic \levy\ processes $Y$ with no Gaussian
  component we have  that 
    $\lim_{t\rightarrow 0}\norm{Y_t}_2/\norm{Y_t}_1=\infty$.
    \end{enumerate}
\end{proposition}
\noindent 
 By the scaling property it is not difficult to show that if 
 $\Lambda$ is a symmetric $\alpha$-stable random measure with
 $\alpha\in (0,2]$,
     then $\h$ satisfies $C_q$  for all $q>0$ when $\alpha=2$ and for
      all $q<\alpha$ when $\alpha<2$.  
For $\alpha<2$ we have the following minor extension: Assume  
$\Lambda$ is induced by a \levy\ process $Y$ with  
\levy\ measure $\nu(dx)=f(x)\,dx$ 
where $f$ is a symmetric function satisfying
$c_1  \abs{x}^{-1-\alpha}\leq f(x)\leq c_2
  \abs{x}^{-1-\alpha}$ for some $c_1,c_2>0$,
   then $\h$ satisfies $C_q$ if and only if
   $q<\alpha$. Proposition~\ref{pro_levy} gives some insight about
   when $C_q$ is satisfied; however, it would be  interesting to develop  more
   general conditions. We postpone the proof of Proposition~\ref{pro_levy} 
 to Section~\ref{proofs}.

\subsection{Results on Integrability of Seminorms}
\label{example}
Let $T$ denote a countable set, $X=\xT$  a real-valued stochastic process and 
$N$ a measurable pseudo-seminorm on $\R^T$ such that $N(X)<\infty$
a.s. For  $X$  Gaussian \cite{Fernique_int}
shows that $e^{\epsilon
  N(X)^2}$ is integrable for some $\epsilon>0$. This result is
extended to Gaussian chaos processes by
\citet[Theorem~4.1]{Borell}. Moreover,  if $X$ is $\alpha$-stable for some
$\alpha\in (0,2)$, \citet[Theorem~3.2]{Arcosta_stable} shows that $N(X)^p$ is
integrable for all $p<\alpha$. When $X$ is infinitely divisible
 \cite{Ros_Sam} provide conditions on the \levy\
measure ensuring integrability of $N(X)$. 
See also \citet{Hoffmann_seminorms} for
further results. 

Given a sequence $(Z_n)_{n\geq 1}$ of independent random variables, 
\citet{Borell_Polynomial} studies, under the condition
\begin{align}
  \label{eq:74}
  \sup_{n\geq 1}
  \frac{\norm{Z_n-E[Z_n]}_q}{\norm{Z_n-E[Z_n]}_2}<\infty, \qquad q\in (2,\infty], 
\end{align} 
 integrability of Banach space valued
random elements which are limits in probability of tetrahedral
polynomials associated with $(Z_n)_{n\geq 1}$. For $q=\infty$,
\eqref{eq:74} is $C_\infty$  but when  $q<\infty$   
\eqref{eq:74} is weaker than $C_\infty$, at least when $(Z_n)_{n\geq 1}$ are
centered random variables. As shown in 
\citet{Borell_Polynomial}, \eqref{eq:74} implies 
equivalence of $L^p$-norms  for  Hilbert space valued
tetrahedral polynomials for $p\leq q$, but not for Banach space valued
tetrahedral polynomials except in the case $q=\infty$. Under the
assumption that  $(Z_n)_{n\geq 1}$ are symmetric random variables satisfying $C_q$,
\citet[Theorem~6.6.2]{Kwapien_R_S} show that we have equivalence of
$L^p$-norms in the above setting. Contrary to
\citet{Borell_Polynomial}, \cite{Kwapien_R_S} and others, we   
consider random elements which  are not necessarily limits of  
tetrahedral polynomials, and also more general spaces are considered. 
This enables us to obtain our integrability results 
for seminorms of stochastic processes.

Weak chaos processes appear in the context of multiple integral
processes; see e.g.\ \citet{Multiple_alpha} for the $\alpha$-stable
case.   Rademacher chaos processes are applied repeatedly 
when studying $U$-statistics; 
see \citet{Pena}. They are also 
used to study  infinitely divisible chaos
processes; see \citet{Suf_con},
\citet{Rosinski_sym}, \citet{Basse_Pedersen} and 
others. Using the results of the present paper,
\cite{Basse_Graversen}
extend some results on Gaussian 
semimartingales (e.g.\ \citet{Gau_Qua} and \citet{Stricker_gl}) 
to a large class of chaos processes.

\section{Main results}
\label{sec_main}
The next lemma, which is a combination of several results, is crucial for this paper.
\begin{lemma}\label{Borell} Let $F$ denote a Banach space and $X$ 
	an $F$-valued tetrahedral polynomial of order $d$ in the independent random variables 
	$Z_1,\dots,Z_n$. Assume that $\h=\{\h_0\}$ satisfies $C_q$ for some
  $q\in (0,\infty]$, where   $\h_{0}=\{Z_1,\dots,Z_n\}$; if $d\geq 2$
  and $q<\infty$ assume moreover  
  that $Z_1,\dots,Z_n$  are symmetric. Then  
  for all $0<p<r\leq q$ with  $r<\infty$  we have that 
  \begin{align}
    \label{eq:97}
    \lVert X\rVert_{L^r(P;F)}\leq k_{p,r,d,\beta}
 \lVert X\rVert_{L^p(P;F)}<\infty,
  \end{align}  where $k_{p,r,d,\beta}$ depends only  on
  $p,q,d$ and the $\beta$'s from $C_q$. If $q=\infty$ and $p\geq 2$ we may choose
  $k_{p,r,d,\beta}=A_d \beta^{2d} r^{d/2}$ with $A_d=2^{d^2/2+2d}$.  
 \end{lemma}
\label{remark_con}
 \noindent  For $q<\infty$ and $d=1$, Lemma~\ref{Borell}  is
a consequence of \citet[2.2.4]{Kwapien_R_S}.
Furthermore, for  $q\in (1,\infty)$ and $d\geq 2$ it is  taken from the proof of
\citet[Theorem~6.6.2]{Kwapien_R_S} and  using
\citet[Remark~6.9.1]{Kwapien_R_S} the result is seen to hold 
  also  for $q\in(0,1]$. For $q=\infty$, Lemma~\ref{Borell} is a consequence of
\citet[Theorem~4.1]{Borell_Polynomial}.  In \citet{Borell_Polynomial} 
the result is only stated 
for $2\leq p<r$, however, a standard
application of H\"older's inequality shows that it is  valid
for all $0<p<r$; see e.g.\ \citet[Lemme~1.1]{Pisier_Borell}. Finally,
in \cite{Borell_Polynomial} there are no  explicit expression for
$A_d$; this can, however, be obtained by applying the next
Lemma~\ref{gen_arg} in the proof of   
\citet[Theorem~4.1]{Borell_Polynomial}. 
\begin{lemma}\label{gen_arg}
  Let $V$ denote a vector space, $N$ a 
  seminorm on $V$, $\epsilon\in (0,1)$ and $x_0,\dots,x_d\in V$.   
  \begin{align}
    \label{eq:15}
    \text{If } N\Big(\sum_{k=0}^d \lambda^k x_k\Big)\leq 1
    \text{ for all }\lambda\in
    [-\epsilon,\epsilon]\quad \text{then}\quad N\Big(\sum_{k=0}^d
    x_k\Big)\leq 2^{d^2/2+d}\epsilon^{-d}.\qquad 
  \end{align} 
  \end{lemma}
\noindent The proof of Lemma~\ref{gen_arg} is postponed to Section~\ref{proofs}.

An $F$-valued  random
element $X$ is said to be a.s.\ separably valued if $P(X\in A)=1$ for
some separable closed subset $A$ of $F$. We  have the following result:
\begin{theorem}\label{int_rade}
  Let   $F$ denote  a metrizable l.c.TVS,  $X\in\wpd(F)$ an 
 a.s.\ separably valued random element and 
  $N$ a lower semicontinuous pseudo-seminorm on $F$
  such that $N(X)<\infty$ a.s. Assume that $\h$ satisfies $C_q$
  for some $q\in (0,\infty]$ and if $q<\infty$ and $d\geq 2$ that all elements 
  in  $\cup_{\xi\in I}\h_\xi$ are 
  symmetric.  Then for all finite $0<p<r\leq q$ we
  have 
    \begin{align}
      \label{eq:45}
      \norm{N(X)}_r\leq  k_{p,r,d,\beta}
     \norm{N(X)}_p<\infty,
    \end{align} where  $k_{p,r,d,\beta}$ depends only  on
  $p,q,d$ and the $\beta$'s from $C_q$. Furthermore,  in the case
  $q=\infty$ we have that  $E[e^{\epsilon 
      N(X)^{2/d}}]<\infty$ for all     $\epsilon< d/(e 2^{d+5}
    \beta_3^4 \norm{N(X)}_2^{2/d})$.  
\end{theorem} 
\noindent For  $q=\infty$, Theorem~\ref{int_rade} answers in the case
where the pseudo-seminorm is lower 
semicontinuous a question raised by
\citet{Borell_Walsh}  concerning integrability of pseudo-seminorms of
Rademacher chaos elements. This additional assumption is satisfied in most
examples, in particular in the examples in \eqref{eq:57}.  
Using the equivalence of norms in Theorem~\ref{int_rade} 
we have by \citet[Corollary~1.4]{Random_multi} the following corollary: 
\begin{corollary}\label{cor_main}
 Let $F$ and $\h$ be as in Theorem~\ref{int_rade} and  $N$ be a 
 continuous seminorm on $F$. Then given $(X_n)_{n\geq
   1}\subseteq  \wpd(F)$ all  a.s.\ 
  separably valued such that  $\lim_n X_n=0$ in
  probability we have $\norm{N(X_n)}_p\rightarrow  0$ for
  all finite $p\in (0,q]$. 
\end{corollary}

Theorem~\ref{int_rade} relies on  the following two lemmas together with
an application of Lemma~\ref{Borell} on the Banach space $l^n_\infty$,
that is $\R^n$ equipped with the sup norm.  First,
arguing  as in \citet[Lemme~1.2.2]{Fernique_book} we have:
\begin{lemma}\label{Fernique_lemma}
  Assume  $F$ is a strongly Lindel\"of l.c.TVS. Then a pseudo-seminorm
  $N$ on $F$ is  lower semicontinuous
 if and only if there exists
   $(x^*_n)_{n\geq 1}\subseteq F^*$ such that $N(x)=\sup_{n\geq 1}
\abs{x_n^*(x)}$ for all $x\in F$. 
\end{lemma}
\begin{proof}
  The  \textit{if}-implication is trivial. To show the \textit{only
    if}-implication  let $A:=\{x\in F:N(x)\leq 1\}$. Then $A$ is convex and  balanced since
  $N$ is a pseudo-seminorm and closed since $N$ is lower
  semicontinuous. Thus by the Hahn-Banach theorem,
  see \citet[Theorem~3.7]{Rudin}, 
  for all $x\notin A$ there exists $x^*\in F^*$ such that
  $\abs{x^*(y)}\leq 1$ for all $y\in A$ and $x^*(y)>1$, 
   showing that 
  \begin{align}
    \label{eq:72}
    A^c=\bigcup_{x\in A^c} \{y\in F:\abs{x^*(y)}>1\}.
  \end{align} Since $F$ is strongly Lindel\"of, there exists 
  $(x_n)_{n\geq 1}\subseteq A^c$   such that
  \begin{align}
    \label{eq:79}
    A^c=\bigcup_{n=1}^\infty \{y\in F:\abs{x^*_n(y)}>1\},
  \end{align} implying that 
  $A=\{y\in F:\sup_{n\geq 1}\abs{x^*_n(y)}\leq 1\}$.
  Thus by homogeneity we have  $N(y)=\sup_{n\geq 
    1} \abs{x^*_n(y)}$ for all $y\in F$.  
\end{proof}
\begin{lemma}\label{gen_exten}
  Let $n\geq 1,\ 0<p<q$  and $C>0$ be given such that 
\begin{align}
  \label{eq:168}
\lVert X\rVert_{L^q(P;l^n_\infty)}\leq C\lVert
X\rVert_{L^p(P;l^n_\infty)}<\infty,\qquad   X\in \ph_{\h_\xi}^d,\ \xi\in I.
\end{align} Then, for all $(X_1,\dots,X_n)\in \pd(\R^n)$ we have that 
\begin{align}
  \label{eq:169}
  \norm{\max_{1\leq k\leq n}\abs{X_k}}_q\leq C\norm{\max_{1\leq k\leq n}\abs{X_k}}_p<\infty.
\end{align}
\end{lemma}
\begin{proof}
Let $X\in \pd(\R^n)$ and choose $(\xi_k)_{k\geq 1}\subseteq I$ and 
$X_{k}\in \ph_{\h_{\xi_k}}^d(\R^n)$ for
 $k\geq 1$ such that $X_k\cdist X$. Moreover, let $U_k=\norm{X_k}_{l_\infty^n}$ 
 and $U=\norm{X}_{l_\infty^n}$. Then, 
 $U_k\cdist U$  showing  that  $(U_k)_{k\geq 1}$ is bounded in $L^0$, 
and by  \eqref{eq:168} and \citet[Corollary~1.4]{Random_multi}, $\{U_k^p:k\geq 1\}$ is
uniformly integrable. This shows that     
  \begin{align}
    \label{eq:172}
    \norm{U}_q\leq \liminf_{k\rightarrow\infty} \norm{U_k}_q\leq
    C\liminf_{k\rightarrow\infty} \norm{U_k}_p=C\norm{U}_p<\infty,
  \end{align} and the proof is complete. 
\end{proof}

\begin{proof}[Proof of Theorem~\ref{int_rade}]
Since $X$ is a.s.\ separably valued we may and will assume that
$F$ is separable. Hence according to 
Lemma~\ref{Fernique_lemma} there exists  $(x_n^*)_{n\geq
   1}\subseteq F^*$ such that $N(x)=\sup_{n\geq 1} \abs{x^*_n(x)}$
for all $x\in F$. For   $n\geq 1$,  let 
 $X_n:=x^*_n(X)$ and 
 $U_n=\sup_{1\leq k\leq n}\abs{X_{k}}$.
 Then $(U_n)_{n\geq 1}$ converges almost surely to
 $N(X)$. For finite  $0<p<r\leq q$  let 
 $C=k_{p,r,d,\beta}$.  Combining  Lemmas~\ref{Borell}  and
 \ref{gen_exten} show
 $ \norm{U_n}_q \leq C\norm{U_n}_p<\infty$ for all $n\geq 1$.
   This implies that  $\{U_n^p:n\geq
 1\}$ is  uniformly integrable and hence we have that
 \begin{align}
   \label{eq:184}
   \norm{N(X)}_r\leq \liminf_{n\rightarrow\infty} \norm{U_n}_r \leq
   C\liminf_{n\rightarrow\infty} \norm{U_n}_p  =C\norm{N(X)}_p<\infty.
 \end{align} Finally, the exponential
  integrability under $C_\infty$ follows by the last part of Lemma~\ref{Borell} since
  \begin{align}
    \label{eq:3}
    E[e^{\epsilon N(X)^{2/d}}]
        \leq 1+\sum_{k=1}^d
    \norm{N(X)}_{2k/d}^{2k/d}+\sum_{k=d+1}^\infty \big(\epsilon 2^{d+5}\beta_3^4
    \norm{N(X)}_2^{2/d}/d\big)^k \frac{k^k}{k!}.
  \end{align} This completes the proof.  
\end{proof}

 Let  $T$ denote a  countable set and  $F=\R^T$ equipped with
 the product topology. $F$ is then a separable and locally convex \fre\
 space and all $x^*\in F^*$
 are of the form  $x\mapsto \sum_{i=1}^n \alpha_i x(t_i)$, for some
 $n\geq 1,\ t_1,\dots ,t_n\in T$ and 
 $\alpha_1,\dots,\alpha_n\in \R$. Thus for $X=(X_t)_{t\in T}$ we have
 that  $X\in\wpd(F)$
 if and only if $X$ is a weak chaos process of order $d$. Rewriting
 Theorem~\ref{int_rade} in the case $F=\R^T$ we obtain 
 the following result:
\begin{theorem}\label{cor_stoc_proc}
 Assume $\h$ satisfies $C_q$
  for some $q\in (0,\infty]$ and if $q<\infty$ and $d\geq 2$ that all elements in 
  $\cup_{\xi\in I}\h_\xi$ are  symmetric. Let $T$ denote a countable set, $\xT$ a weak
  chaos process of order $d$ and  $N$ a lower semicontinuous pseudo-seminorm on $\R^T$
   such that $N(X)<\infty$ a.s.\
  Then  for   all  finite $0<p<r\leq q$ we have
  \begin{align}
    \label{eq:148}
    \norm{N(X)}_r\leq  k_{p,r,d,\beta}
     \norm{N(X)}_p<\infty,
  \end{align}
  and in the case $q=\infty$ that $E[e^{\epsilon
      N(X)^{2/d}}]<\infty$   for all $\epsilon<d /(e 2^{d+5} \beta_3^4
    \norm{N(X)}_2^{2/d})$.  
\end{theorem}
\noindent For example, let $T=[0,1]\cap \Q$,
$(X_t)_{t\in T}$ be of the form
$X_t=\int_0^1 f(t,s)\,dY_s$ where $Y$ is a symmetric
normal inverse Gaussian L\'evy process, and 
$\morf{N}{\R^T}{[0,\infty]}$ be given by \eqref{eq:57}.    
 Then, $N$ is a lower semicontinuous pseudo-seminorm and $X$ is weak chaos
 process of order one satisfying $C_q$ for all $q<1$ according to
 Proposition~\ref{pro_levy}. Thus, if 
  $N(X)<\infty$ a.s.\ then   $E[N(X)^p]<\infty$ for all $p<1$,
  according to Theorem~\ref{cor_stoc_proc}. 
  
Let $\G$ denote a vector space of Gaussian random variables  
and  $\pid(\R)$ be the closure in
probability of the random variables $p(Z_{1},\dots,
Z_{n})$,  where $n\geq 1$,  $Z_1,\dots,Z_n\in \G$
and $\morf{p}{\R^n}{\R}$ is a polynomial of degree at most $d$ (not
necessary tetrahedral). 
\begin{lemma}\label{lem_gau}
  Let $F$ be a l.c.TVS and $X$ an $F$-valued random element such that 
$x^*(X)\in \pid(\R)$ for all $x^*\in F^*$; then $X\in
weak\text{-}\pd(F)$ where  $\h=\{\h_0\}$ and $\h_0$ is a Rademacher sequence. 
\end{lemma}
\noindent Recall that a sequence of independent, identically
distributed random variables 
$(Z_n)_{n\geq 1}$ such that $P(Z_1=\pm 1)=1/2$ is called a Rademacher sequence.
\begin{proof}
  Let $n\geq 1,\  x^*_1,\dots,x^*_n\in F^*$ and 
$W=(x^*_1(X),\dots,x_n^*(X))$. 
We need to show that $W\in \pd(\R^n)$. 
For all $k\geq 1$ we may choose polynomials 
$\morf{p_{k}}{\R^k}{\R^n}$ of degree at most $d$
and $Y_{1,k},\dots, Y_{k,k}$  
 independent  standard normal random variables such that with $
Y_k=(Y_{1,k},\dots, Y_{k,k})$ we have 
$\lim_k  p_k( Y_k)=W$ in probability.  Hence it  suffices to show
$p_k( Y_k)\in \pd(\R^n)$    for all $k\geq 1$. Fix $k\geq 1$ and let us  write $p$ and $Y$ 
for $p_k$ and $Y_k$. Reenumerate
 $\h_0$ as $k$ independent Rademacher sequences $(Z_{i,m})_{i\geq 1}$
 $m=1,\dots,k$  and set
\begin{equation}
  \label{eq:83}
  U_{j}=\frac{1}{\sqrt{j}}\sum_{i=1}^j (Z_{1,i},\dots,Z_{k,i}),\qquad
  j\geq 1.
\end{equation}
Then, by the central limit
theorem $ U_j\cdist  Y$ and hence 
$ p(U_j)\cdist  p( Y)$. Due to the fact that  all $Z_{i,m}$ only 
takes on the values $\pm 1$, 
$p(U_j)\in \ph_{\h_0}^d(\R^n)$ for all $j\geq 1$, showing  that $p( Y)\in \pd(\R^n)$. 
\end{proof}

The $\h$ in Lemma~\ref{lem_gau} trivially satisfies $C_\infty$ with
$\beta_3=1$ and hence a combination of 
Theorem~\ref{int_rade} and Lemma~\ref{lem_gau} shows:

\begin{proposition}\label{cor_gau}
Let $F$ be a l.c.TVS and $X$ an a.s.\ separably valued random element
in $F$ such that $x^*(X)\in \pid(\R)$ for all $x^*\in F^*$. 
Then, for all   lower 
semicontinuous pseudo-seminorms $N$  on $F$ satisfying
$N(X)<\infty$ a.s.\  we have 
    \begin{equation}
      \label{eq:152}
   \norm{N(X)}_r\leq    2^{d^2/2+d} 
\left(\frac{r-1}{p-1}\right)^{d/2}\norm{N(X)}_p<\infty,
    \end{equation} and 
    $E[e^{\epsilon N(X)^{2/d}}]<\infty$ for all $\epsilon<d /(e 2^{d+5} 
    \norm{N(X)}_2^{2/d})$.
\end{proposition}
\noindent The integrability of $e^{\epsilon  N(X)^{2/d}}$ for some $\epsilon>0$
is a consequence of the seminal work \citet[Theorem~4.1]{Borell}.  
However, the above provides a very simple proof of this result and 
gives also   equivalence of  $L^p$-norms and explicit constants. When
$F=\R^T$ for some countable set $T$, Proposition~\ref{cor_gau} covers
processes $X=(X_t)_{t\in T}$,  where all time variables have the 
following representation in terms of multiple
Wiener-It\^o integrals with respect to a Brownian 
motion $W$,
\begin{equation}
  \label{eq:102}
  X_t=\sum_{k=0}^d \int_{\R^k_+} f(t,k;s_1,\dots,s_k)\,dW_{s_1}\cdots
  dW_{s_k},\qquad t\in T.
\end{equation} 

The next result is known from \citet[Theorem~3.1]{Arcones} for general
Gaussian polynomials. 
\begin{proposition}\label{pro_series}
Assume that $\h=\{\h_0\}$ satisfies $C_q$ for
some $q\in [2,\infty]$ and $\h_0$ consists of symmetric random
variables. Let $F$ denote a Banach space 
and  $X$ an a.s.\ separably valued random element in $F$
with $x^*(X)\in \pd(\R)$ for all $x^*\in
F^*$. Then there
exists $x_0,x_{i_1,\dots,i_k}\in F$ and 
$\{Z_n:n\geq 1\}\subseteq \h_0$ such that for all finite $p\leq q$ 
  \begin{equation}
    \label{eq:71}
    X=\lim_{n\rightarrow\infty}\Big(x_0+\sum_{k=1}^d 
    \sum_{1\leq i_1<\dots<i_k\leq n} x_{i_1,\dots,i_k}
  \prod_{j=1}^k Z_{{i_j}}\Big)\qquad \text{a.s.\ and in }L^p(P;F). 
  \end{equation} 
\end{proposition}
\begin{proof} We follow \citet[Lemma~3.4]{Arcones}. Since $X$ is a.s.\
  separably valued 
  we may and do assume  $F$ is separable, which implies that  $F^*_1:=\{x^*\in
  F^*:\norm{x^*}\leq 1\}$ is
metrizable and compact in the weak*-topology by the Banach-Alaoglu theorem; see
\citet[Theorem~3.15+3.16]{Rudin}. 
  Moreover, the map $x^*\mapsto x^*(X)$ from
$F^*_1$ into $L^0$ is trivially weak*-continuous and
thus  a weak*-continuous map into $L^2$ by
Corollary~\ref{cor_main}. This  shows that   
$\{x^*(X):x^*\in F^*_1\}$ is compact in $L^2$ and hence separable. By definition
of $\pd(\R)$,  this implies that there exists a countable set
$\{Z_n:n\geq 1\}\subseteq \h_0$ such that 
\begin{align}
  \label{eq:183}
  x^*(X)=\sum_{A\in N_d} a(A,x^*)Z_A,\qquad \text{in }L^2, 
\end{align} for some $a(A,x^*)\in \R$, where   $N_d=\{A\subseteq \N:\abs{A}\leq
d\}$ and $Z_A=\prod_{i\in A}Z_{i}$ for $A\in N_d$. For  $A\in
N_d$, the map  $x^*\mapsto a(A,x^*)$ from 
$F^*$ into $\R$ is linear and 
weak*-continuous and hence there exists $x_A\in F$ such that
$a(A,x^*)=x^*(x_A)$, showing that  
\begin{align}
  \label{eq:186}
  x^*(X)=\lim_{n\rightarrow\infty} x^*\Big( \sum_{A\in N_d^n} x_A
  Z_A\Big),\qquad \text{in }L^2, 
\end{align} where   $N_d^n=\{A\in  N_d:A\subseteq \{1,\dots,n\}\}$.
Since $F$ is separable, \eqref{eq:186} and 
\citet[Theorem~6.6.1]{Kwapien_R_S}  show that 
\begin{align}
  \label{eq:4}
  \lim_{n\rightarrow\infty}\sum_{A\in N_d^n} x_A Z_A= X\qquad
  \text{a.s.} 
\end{align} By Corollary~\ref{cor_main} the convergence
also takes place in $L^p(P;F)$ for all finite $p\leq q$, which
completes the proof. 
\end{proof}
The above proposition gives rise  to the following corollary: 
\begin{corollary}\label{cor_series}
  Assume that $\h=\{\h_0\}$ satisfies $C_q$ for
some $q\in [2,\infty]$ and $\h_0$ consists of symmetric random
variables. Let $T$ denote a set, $V(T)\subseteq \R^T$ a
separable Banach space where the map
$f\mapsto f(t)$ from $V(T)$ into $\R$ is continuous for all $t\in
T$, and $X=(X_t)_{t\in T}$ a  stochastic process with sample paths in $V(T)$
satisfying  $X_t\in\pd(\R)$ for all $t\in T$. Then there exists 
$x_0,x_{i_1,\dots,i_k}\in V(T)$ and 
$\{Z_n:n\geq 1\}\subseteq \h_0$ such that
  \begin{equation}
    \label{eq:31}
    X=\lim_{n\rightarrow\infty}\Big(x_0+\sum_{k=1}^d\sum_{1\leq
      i_1<\dots<i_k\leq n} x_{i_1,\dots,i_k} 
  \prod_{j=1}^k Z_{{i_j}}\Big)
  \end{equation}a.s.\ in $V(T)$ and in $L^p(P;V(T))$ for all finite
  $p\leq q$. 
\end{corollary}
  \begin{proof}
    For $t\in T$, let $\morf{\delta_t}{V(T)}{\R}$ denote the map
    $f\mapsto f(t)$. Since $V(T)$ is a separable Banach space and 
    $\{\delta_t:t\in T\}\subseteq V(T)^*$ 
    separate points in $V(T)$ we have
    \begin{enumerate}[(i)]
    \item \label{cor_1}the Borel $\sigma$-field  on
 $V(T)$ equals the cylindrical $\sigma$-field
 $\sigma(\delta_t:t\in T)$,
\item \label{cor_2} $\{\sum_{i=1}^n\alpha_i\delta_{t_i}:\alpha_i\in
  \R,\ t_i\in T,\ n\geq 1\}$  is 
sequentially weak*-dense in $V(T)^*$,
    \end{enumerate} see e.g.\ \citet[page~287]{Rosinski}. 
 By \eqref{cor_1} we may  regard $X$ as a random element
in $V(T)$ and  by \eqref{cor_2}  it follows that   $x^*(X)\in \pd(\R)$ for all
$x^*\in V(T)^*$. Hence the result is a consequence of
Proposition~\ref{pro_series}.  
  \end{proof}
\noindent \citet[Theorem~5.1]{Borell_Polynomial}
shows Corollary~\ref{cor_series} assuming \eqref{eq:74},  $T$ is a compact metric
  space, $V(T)=C(T)$ and $X\in L^q(P;V(T))$.
By assuming  $C_q$ instead of the weaker
condition \eqref{eq:74}  we can  omit the assumption $X\in L^q(P;V(T))$. 
Note also that   by Theorem~\ref{cor_stoc_proc}  the last assumption is
satisfied under $C_q$. When $\h_0$ consists of symmetric $\alpha$-stable
random variables and $d=1$, Corollary~\ref{cor_series} is known from
\citet[Corollary~5.2]{Rosinski}. The separability assumption on
$V(T)$ in Corollary~\ref{cor_series} is crucial. Indeed, for all
$p>1$, \citet[Proposition~4.5]{Jain_Monrad_B_p} construct a separable
 centered Gaussian process $X=(X_t)_{t\in [0,1]}$ with sample paths in
 the non-separable Banach space 
 $B_p$ of functions of finite $p$-variation on $[0,1]$  such that  the range of $X$ is a 
non-separable subset of $B_p$ and hence the conclusion in
Corollary~\ref{cor_series} can not be true. However, for the
non-separable Banach space $B_1$ a result similar to
Corollary~\ref{cor_series} is shown in \cite{Gau_Qua} for Gaussian
processes,  and extended to weak chaos processes in
\cite{Basse_Graversen}. 

\section{Proofs of Proposition~\ref{pro_levy} and Lemma~\ref{gen_arg}}
\label{proofs}
Let us start by proving Proposition~\ref{pro_levy}.
\begin{proof}[Proof of Proposition~\ref{pro_levy}]
 Assume that $\Lambda$ is a random measure induced by a L\'evy process
 $Y=(Y_t)_{t\in [0,T]}$. For arbitrary $A\in \s$ let  $Z=\Lambda(A)$. 

To prove the \textit{if}-implication of  \eqref{IG} let $q\in
 (0,\frac{1}{2})$ and assume that $Y_1\dist \text{IG}(\mu,\lambda)$. Then
 $Z\dist\text{IG}( m(A)\mu, m(A)^2\lambda)$, 
 where $m$ is the Lebesgue measure, and hence 
 with  $c_Z= m(A)^2\lambda $ we have that 
$Z/c_Z\dist \text{IG}(\mu/(\lambda m(A)),1)$, which has 
a density which on $[1,\infty)$ is  bounded from below and above by constants (not
depending on $x$) times $g_Z(x)$, where 
\begin{equation}
  \label{eq:55}
 \morf{g_Z}{\R_+}{\R_+},\qquad x\mapsto x^{-3/2}\exp[-x  (\lambda m(A))^2 /(2\mu^2)].
\end{equation} 
 Thus there exists a constant
$c>0$, not depending on $A$ or $s$,   such that 
\begin{align}
  \label{eq:91}
  \frac{E[\abs{Z/c_Z}^q,\abs{Z/c_Z}>s]}{s^q P(\abs{Z/c_Z}>s)} \leq c
  \sup_{u>0}\left(\frac{\int_u^\infty x^{q-3/2} e^{-x}\,dx
  }{u^q \int_u^\infty x^{-3/2} e^{-x}\,dx  } \right) \qquad
s\geq 1. 
\end{align} Using e.g.\ l'H\^opital's rule it is easily seen 
that \eqref{eq:91} is finite, showing \eqref{eq:139}. 
Therefore $C_q$ follows by the inequality 
\begin{equation}
  \label{eq:156}
  P( Z/c_Z\geq 1)\geq \frac{e^{-1/2}}{\sqrt{2\pi}}\int_1^\infty 
  x^{-3/2}\exp[-x(\lambda T)^2 /(2\mu^2)]\,dx.
\end{equation}  To show the
\textit{only if}-implication of \eqref{IG} note that $n^{2} Y_{1/n} \cdist X$ as
$n\rightarrow \infty$, where $X$
follows a $\frac{1}{2}$-stable distribution on $\R_+$. 
Assume that  $\h$ satisfies $C_q$ for some $q\geq
1/2$. Then, by Remark~\ref{eqi_norms} there exists $c>0$ such that
$\norm{Y_t}_{1/2}\leq c \norm{Y_t}_{1/4}$ for all $t\in [0,1]$, and since
$\{n^{2} Y_{1/n}:n\geq 1\}$ is bounded in $L^0$ it
is also bounded in $L^{1/2}$. But this contradicts 
\begin{equation}
  \label{eq:154}
  \infty=\norm{X}_{1/2}\leq \liminf_{n\rightarrow \infty} \norm{n^2 Y_{1/n}}_{1/2},
\end{equation} and shows that $\h$ does not satisfy $C_q$.

To show the \textit{if}-implication of \eqref{NIG} assume that $Y_1\dist
\text{NIG}(\alpha,0,0,\delta)$. Then, 
$Z\dist\text{NIG}(\alpha,0,0,m(A)\delta)$ and  
with  $c_Z=m(A)\delta$ we have that $Z/c_Z\dist \epsilon U^{1/2}_Z$, 
where $U_Z$ and $\epsilon$ are
independent, $U_Z\dist \text{IG}(1/(m(A) \delta\alpha),1 )$
and $\epsilon \dist \text{N}(0,1)$. 
For $q\in (0,1)$, 
\begin{align}
  \label{eq:98}
  E[\abs{Z/c_Z}^q,\abs{Z/c_Z}>s]=
 \sqrt{2\pi\inv} \Big( {}& \int_0^s
    E[\abs{x U^{1/2}_Z}^q,\abs{x U^{1/2}_Z}>s]e^{-x^2/2}\,dx \\ {}& +
\int_s^\infty
    E[\abs{x U^{1/2}_Z}^q,\abs{x U^{1/2}_Z}>s]e^{-x^2/2}\,dx \Big).
\end{align} Using  the above \eqref{IG} on $U_Z$ and $q/2$,  there exists a constant
$c_1>0$ such that 
\begin{align}
  \label{eq:99}
  \int_0^s  E[\abs{x U^{1/2}_Z}^q,\abs{x U^{1/2}_Z}>s]e^{-x^2/2}\,dx
\leq {}& c_1 s^q
\int_0^s P\big(U_Z>(s/x)^2 \big) e^{-x^2/2}\,dx\\ \leq  c_1 s^q
\int_0^\infty  P\big(x U^{1/2}_Z>s\big)
e^{-x^2/2}\,dx= {}& c_1 \sqrt{\pi 2\inv} s^q  P(\abs{Z/c_Z}>s).
\end{align} Furthermore, it well known that 
there exists a constant  $c_2>0$ such that for all $s\geq 1$ 
\begin{align}
  \label{eq:100}
{}&   \int_s^\infty
    E[\abs{x U^{1/2}_Z}^q,\abs{x U^{1/2}_Z}>s]e^{-x^2/2}\,dx 
    \\  \leq {}&   E[U_Z^{q/2}] \int_s^\infty x^q 
  e^{-x^2/2}\,dx \leq c_2 s^q E[U_Z^{q/2}] \int_s^\infty 
  e^{-x^2/2}\,dx. 
    \end{align} 
    Since $U_Z$ has a density given by \eqref{eq:146} it is easily
    seen that
    \begin{equation}
      \label{eq:155}
      E[U_Z^{q/2}]\leq 1+\frac{1}{\sqrt{2\pi}}\int_1^\infty
    x^{q/2-3/2}\,dx. 
    \end{equation}
Moreover, using that $Z/c_Z\dist \text{NIG}( m(A)\alpha \delta, 0,0,1)$ and
that $K_1(z)\geq e^{-z}/z$ for all $z>0$, 
it is not difficult to show that there 
    exists a constant $c_3$, not depending 
    on $s$ and  $A$, such that
    \begin{align}
      \label{eq:145}
      \int_s^\infty e^{-x^2/2}\,dx\leq c_3 P(\abs{Z/c_Z}>s),\qquad \text{for
        all } s\geq 1.
    \end{align} By combining the above we  obtain \eqref{eq:139}  
   and by \eqref{eq:145} applied  on $s=1$, $C_q$ follows.  The \textit{only
      if}-implication of \eqref{NIG} follows similar to the one of \eqref{IG}, now
    using that $(n^{-1}Y_{1/n})_{n\geq 1}$ converge weakly
    to a symmetric $1$-stable distribution. 
    \eqref{not_4} is a consequence of  the next lemma.
  \end{proof} 
  
The following lemma is concerned with the dynamics of the first and
second moments of \levy\
processes, and it has  Proposition~\ref{pro_levy}~\eqref{not_4} as 
a direct consequence.
\begin{lemma}
  Let $Y$ denote a non-deterministic and square-integrable \levy\ 
  process with no Gaussian
  component.  Then
  $\norm{Y_t}_1=o(t^{1/2})$ and $\norm{Y_t}_2\sim t^{1/2}
  \sqrt{E[(Y_1-E[Y_1])^2]}$ as $t\rightarrow 0$. 
\end{lemma}
\begin{proof}
We have 
\begin{equation}
  \label{eq:150}
  E[Y_t^2]=Var(Y_t)+E[Y_t]^2=Var(Y_1) t+E[Y_1]^2t^2, 
\end{equation} which shows that $\norm{Y_t}_2\sim t^{1/2}Var(Y_1)^{1/2}$
as $t\rightarrow 0$.

  To show  that $\norm{Y_t}_1=o(t^{1/2})$ as $t\rightarrow 0$ we may 
  assume that $Y$ is
  symmetric. Indeed  let $\mu=E[Y_1]$,
 $Y'$ an independent copy of $Y$ and $\tilde Y_t=Y_t-Y_t'$.
Then  $\tilde Y$ is a symmetric square-integrable \levy\ process and 
 \begin{equation}
   \label{eq:111}
   \norm{Y_t}_1\leq \norm{Y_t-\mu t}_1+ \abs{\mu}\leq
   \norm{Y_t-\mu t-(Y_t'-\mu t)}_1+ \abs{\mu} t=  \norm{\tilde Y_t}_1+ \abs{\mu} t. 
 \end{equation}
Hence assume that $Y$ is symmetric.  Recall, e.g.\ from
\citet[Exercise~5.7]{JHJ_Book}, that for any random variable $U$ we
have 
\begin{equation}
  \label{eq:114}
 \norm{U}_1=\frac{1}{\pi}\int \frac{1-\Re \phi_{U}(s)}{s^2}\,ds,
\end{equation} where $\phi_U$ denotes the characteristic function of
$U$. Using the inequalities  $1-e^{-x}\leq 1\wedge x$ 
and $1-\cos(x)\leq 4(1\wedge  x^2)$ for all $x\geq 0$ it follows that 
  with $\psi(s):=4 \int (1\wedge \abs{sx}^2)\,\nu(dx)$ we have 
\begin{equation}
  \label{eq:113}
 \norm{Y_t}_1\leq \frac{1}{\pi}\int \frac{1-e^{-t\psi(s)}}{s^2}\,ds\leq \frac{1}{\pi}\int
  \frac{\abs{t\psi(s)}\wedge 1}{s^2}\,ds. 
\end{equation} Note that $\psi(s)<\infty$ since $Y$ is
square-integrable. By substitution we get  
\begin{align}
  \label{eq:37}
  \int   \frac{\abs{ t\psi(s)}\wedge 1}{s^2}\,ds\leq 2 t^{1/2} \int_0^\infty
  \frac{\abs{t\psi(t^{-1/2}s)}\wedge 1}{s^2}\,ds. 
\end{align} Hence to complete the proof we need only to show that 
\begin{align}
    \label{eq:124}
     \lim_{t\rightarrow 0} \int_0^\infty
  \frac{\abs{t\psi(t^{-1/2}s)}\wedge 1}{s^2}\,ds =0.
  \end{align}
 Setting $c=4 \int x^2\,\nu(dx)$ we have for all  $\epsilon>0$ 
  \begin{align}
       {}&  \limsup_{t\rightarrow 0} \int_0^\infty
  \frac{\abs{t\psi(t^{-1/2}s)}\wedge 1}{s^2}\,ds \\ \label{eq:117} \leq{}&
  \limsup_{t\rightarrow 0} \int_0^{\epsilon/c} 
  \frac{\abs{t\psi(t^{-1/2}s)}\wedge 1}{s^2}\,ds
  + \limsup_{t\rightarrow 0} \int_{\epsilon/c}^\infty  
  \frac{\abs{t\psi(t^{-1/2}s)}\wedge 1}{s^2}\,ds.   
  \end{align} Using that   $\psi(x)\leq c x^2$ for  $x\geq 0$ we get
  \begin{align}
    \label{eq:120}
     \limsup_{t\rightarrow 0} \int_0^{\epsilon/c} 
  \frac{\abs{t\psi(t^{-1/2}s)}\wedge 1}{s^2}\,ds\leq\epsilon. 
  \end{align} On the other hand, Lebesgue's dominated convergence
  theorem shows that   
  \begin{align}
    \label{eq:115}
    \psi(x)x^{-2} =4\int (x^{-2}\wedge
    s^2)\,\nu(dx)\xrightarrow[x\rightarrow\infty]{} 0,
  \end{align} implying that   $ t\psi(t^{-1/2}s)\rightarrow 0$ as
  $t\rightarrow 0$ for all $s\geq 0$. Thus another application of Lebesgue's
  dominated convergence theorem yields 
  \begin{align}
    \label{eq:121}
    \limsup_{t\rightarrow 0} \int_{\epsilon/c}^\infty  
  \frac{\abs{t\psi(t^{-1/2}s)}\wedge 1}{s^2}\,ds=0, 
  \end{align} which by \eqref{eq:117} and \eqref{eq:120} shows \eqref{eq:124}.
  \end{proof}
Let us proceed with the proof of Lemma~\ref{gen_arg}. 
  \begin{proof}[Proof of Lemma~\ref{gen_arg}]
        Assume first that $x_0,\dots,x_d\in \R$. By induction
        in $d$,  let us show:
    \begin{align}
      \label{eq:18}
     \text{If }  \Big\vert\sum_{k=0}^d \lambda^k x_k\Big\vert\leq 1\text{ for all }\lambda\in
    [-\epsilon,\epsilon]\quad  \text{then}\quad   \Big\vert\sum_{k=0}^d
    x_k\Big\vert\leq 2^{d^2/2+d}\epsilon^{-d}. 
    \end{align} For $d=1,2$  \eqref{eq:18}
    follows by a straightforward argument, so 
    assume $d\geq 3$, \eqref{eq:18} holds for $d-1$ and that the left-hand side
    of  \eqref{eq:18} holds for $d$. We have 
     \begin{align}
      \label{eq:67}
       \Big|\sum_{k=0}^d \lambda^k (\epsilon^k x_k)\Big|\leq 1,\qquad
       \text{for all }\lambda\in [-1,1],
    \end{align} which  by \citet[Aufgabe~77]{Polya} shows that 
    $\abs{x_d\epsilon^d}\leq 2^{d}$ and hence   $\abs{x_d}\leq
    2^{d}\epsilon^{-d}$.  For $\lambda\in [-\epsilon,\epsilon]$,
  the triangle inequality yields 
        \begin{align}
      \label{eq:68}
    \Big|\sum_{k=0}^{d-1} \lambda^k  x_k\Big|\leq 1 +2^d,\quad
    \text{and hence}\quad \Big|\sum_{k=0}^{d-1} \lambda^k 
      \frac{x_k}{1+2^d}\Big|\leq 1.
    \end{align}  The induction hypothesis implies 
    \begin{align}
      \label{eq:51}
      \Big|\sum_{k=0}^{d-1} x_k\Big|\leq \epsilon^{-(d-1)}2^{(d-1)^2+(d-1)}(1+2^d),
    \end{align} and hence another application of the triangle inequality
    shows that 
    \begin{align}
      \label{eq:52}
      \Big|\sum_{k=0}^{d} x_k\Big|\leq{}& \epsilon^{-d}
      2^d+\epsilon^{-(d-1)}2^{(d-1)^2/2+(d-1)}(1+2^d)\\ \leq{}&
      \epsilon^{-d}2^{d^2/2+d}\Big(2^{-d^2/2}+2^{-1/2-d}+2^{-1/2}\Big),
    \end{align} which is less than or equal to
    $\epsilon^{-d}2^{d^2/2+d}$ since $d\geq 3$. This completes  the proof of
    \eqref{eq:18}.

    Now let $x_0,\dots,x_d\in V$. Since $N$ is a seminorm,  Hahn-Banach theorem
    (see \citet[Theorem~3.2]{Rudin}) shows that  there exists a family $\Lambda$ of linear 
    functionals on $V$  such that 
    \begin{align}
      \label{eq:21}
      N(x)=\sup_{F\in \Lambda} \abs{F(x)},\qquad \text{for all }x\in V. 
    \end{align}
    Assuming  that the left-hand side of \eqref{eq:15} is satisfied we
    have 
        \begin{align}
      \label{eq:42}
        \Big|\sum_{k=0}^d
           \lambda^k F(x_k)\Big|\leq 1,\qquad \text{for all }  \lambda\in
         [-\epsilon,\epsilon] \text{ and all } F\in \Lambda, 
    \end{align} which by \eqref{eq:18} shows 
    \begin{align}
      \label{eq:43}
      \Big|F\Big(\sum_{k=0}^d  x_k\Big)\Big|=\Big|\sum_{k=0}^d
      F(x_k)\Big|\leq 2^{d(d-1)} \epsilon^{-d},\qquad \text{for all }
      F\in \Lambda.
    \end{align} This completes the proof.
  \end{proof}

\end{document}